\documentclass{article}
\usepackage{maa-monthly}

\theoremstyle{plain}
\newtheorem{theorem}{Theorem}
\newtheorem{lemma}[theorem]{Lemma}
\newtheorem{corollary}[theorem]{Corollary}

\newtheorem{conjecture}[theorem]{Conjecture}
\newtheorem{question}[theorem]{Question}
\theoremstyle{definition}
\newtheorem{definition}[theorem]{Definition}

\begin{document}

\title{An Equilateral Triangle of Side $> n$ Cannot be Covered by \(n^2 + 1\) Unit Equilateral Triangles Homothetic to it}
\markright{Covering a triangle with unit-sided triangles}

\author{Jineon Baek and Seewoo Lee}

\maketitle

\begin{abstract}
John Conway and Alexander Soifer showed that an equilateral triangle \(T\) of side slightly longer than $n$ can be covered by \(n^2 + 2\) unit equilateral triangles.
They also conjectured that it is impossible to cover $T$ with \(n^2 + 1\) unit equilateral triangles, no matter how close the side of \(T\) is to $n$.

While the Conway--Soifer conjecture remains open, we prove an important case where the sides of the triangles used for covering are parallel to the sides of $T$ (e.g., $\bigtriangleup$ and $\bigtriangledown$).
That is, we show that if all unit equilateral triangles are required to be homothetic to \(T\), 
then the minimum number of unit equilateral triangles that can cover \(T\) of side slightly longer than $n$ is exactly \(n^2 + 2\).

Our proof generalizes to covering $T$ by (not necessarily equilateral) triangles of base one parallel to the $x$-axis and height equal to that of a unit equilateral triangle.
Using our method, we also determine the largest side length $n + 1/(n + 1)$ (resp. $n + 1 / n$) of $T$ such that the equilateral triangle $T$ can be covered by $n^2+2$ (resp. $n^2 + 3$) unit equilateral triangles homothetic to $T$.
\end{abstract}

\section{Introduction.}

John Conway and Alexander Soifer showed that $n^2 + 2$ unit equilateral triangles can cover an equilateral triangle $T$ of side $> n$ by providing two coverings (Figures \ref{fig:triangle1} and \ref{fig:triangle2}) \cite{conway2004coverup, conway2005covering}.
Their paper was famously short as an attempt to set the world record for the shortest math paper ever; see \cite{soifer2009coffee} for the full story by the second author.
We provide a detailed explanation of their constructions in Section \ref{sec:description}.

\begin{figure}[h]
   \begin{minipage}{0.5\textwidth}
     \centering
     \includegraphics[width=0.9\linewidth]{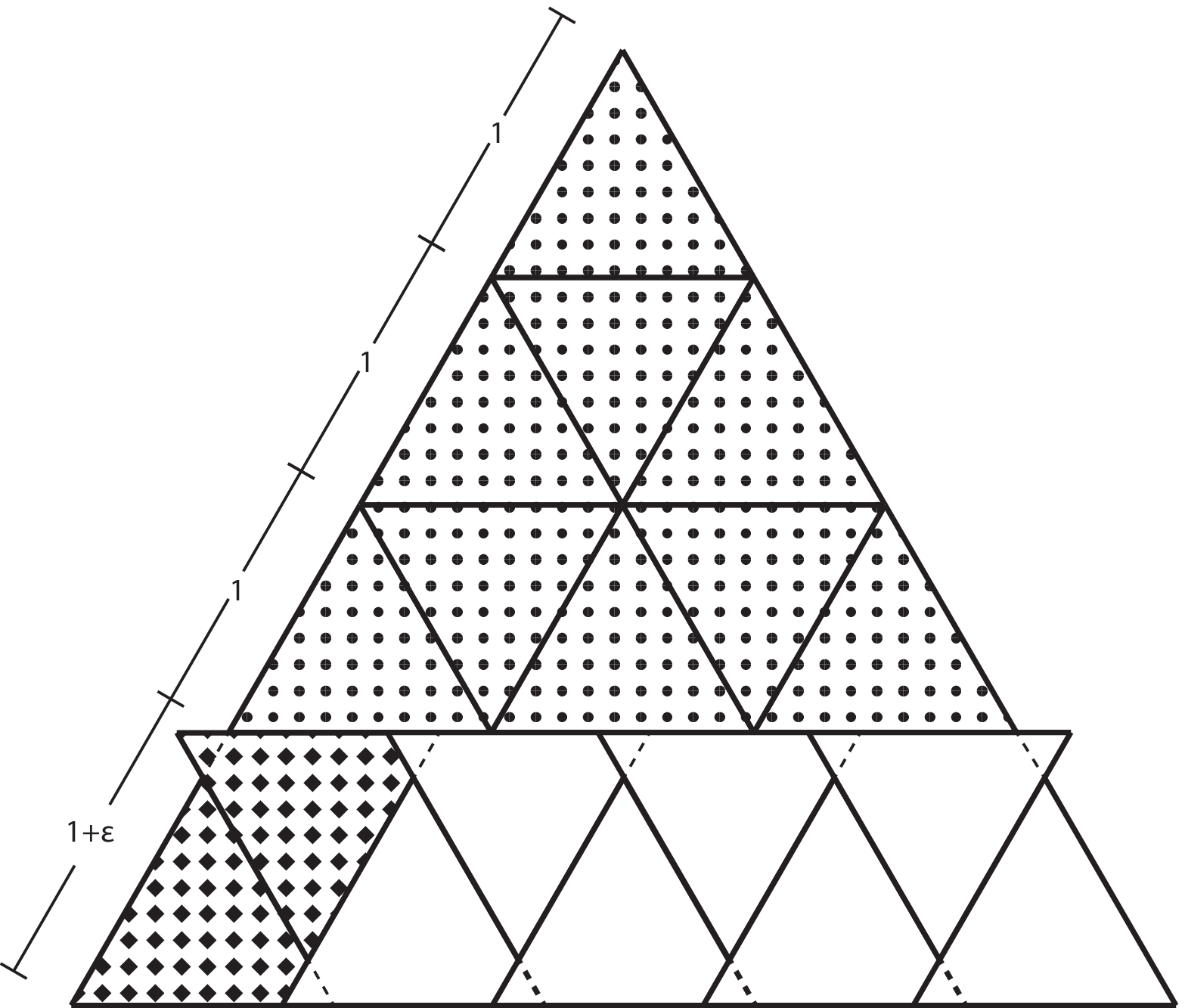}
     \caption{}
     \label{fig:triangle1}
   \end{minipage}\hfill
   \begin{minipage}{0.5\textwidth}
     \centering
     \includegraphics[width=0.9\linewidth]{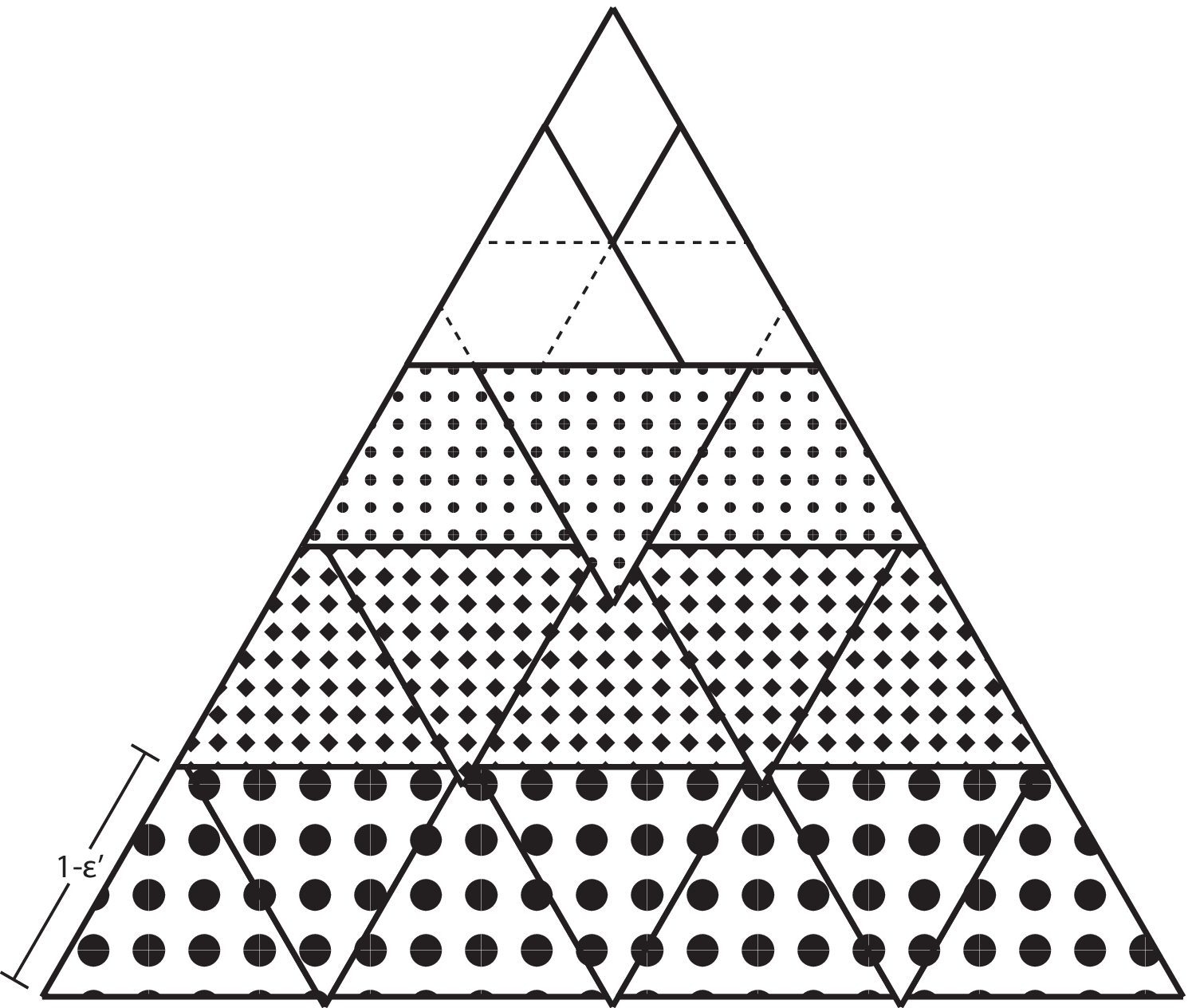}
     \caption{}
     \label{fig:triangle2}
   \end{minipage}
\end{figure}

\begin{theorem}[Conway and Soifer \cite{conway2004coverup, conway2005covering}]
\label{thm:main-prev}
$n^2 + 2$ unit equilateral triangles can cover an equilateral triangle $T$ of side $n + \varepsilon$ for a sufficiently small $\varepsilon > 0$.
\end{theorem}

In the same work, they also conjectured that $n^2 + 1$ unit equilateral triangles cannot cover any equilateral triangle $T$ of side $> n$ \cite{conway2005covering}.

\begin{conjecture}[Conway and Soifer \cite{conway2004coverup}]
\label{conj:main}
$n^2 + 1$ unit equilateral triangles cannot cover an equilateral triangle $T$ of side $n + \varepsilon$ for any $\varepsilon > 0$.
\end{conjecture}

To get a feeling for Conjecture \ref{conj:main}, we follow \cite{soifer2010building} and prove the case $n = 2$.
That is, we need at least six unit equilateral triangles to cover an equilateral triangle $T$ of side $> 2$.
Take the six points consisting of the vertices of $T$ and their midpoints as in (a) of Figure \ref{fig:nis2case}.
Then, we need at least six triangles to cover $T$ since each unit equilateral triangle can cover at most one point.
The constructions by Conway and Soifer, as in (b) and (c) of Figure \ref{fig:nis2case}, show how to cover $T$ with six triangles.

\begin{figure}[h]
    \centering
    \includegraphics[width=0.9\linewidth]{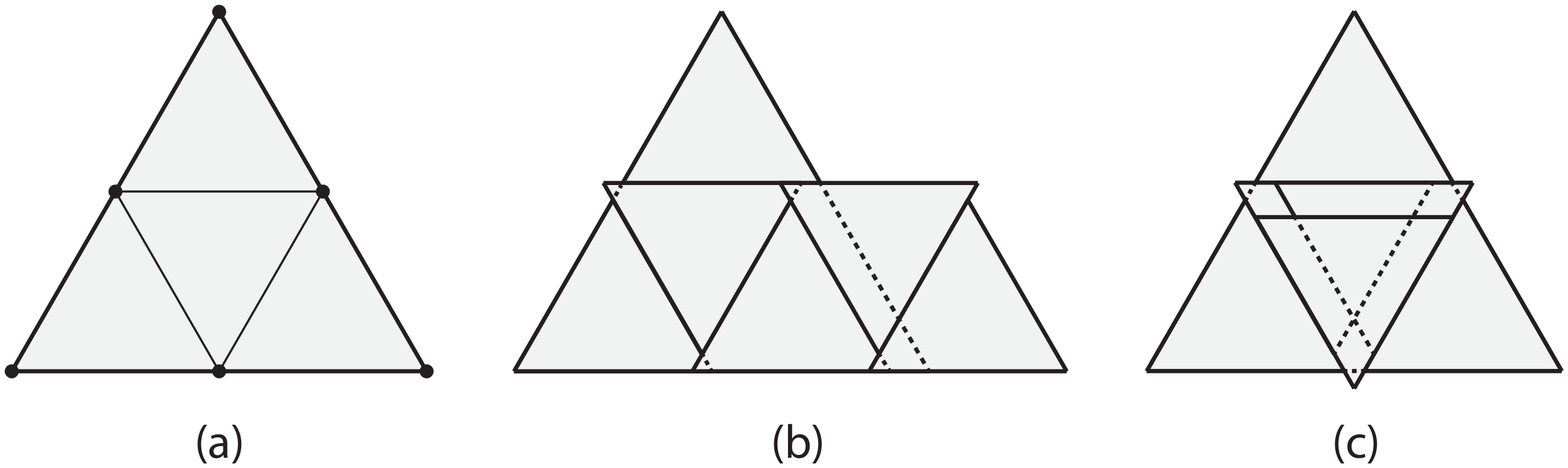}
    \caption{}
    \label{fig:nis2case}
\end{figure}

Dmytro Karabash and Soifer showed that for every \emph{non-equilateral} triangle \(T\), \(n^2 + 1\) triangles similar to \(T\) and with the ratio of linear sizes \(1: (n + \varepsilon)\) can cover \(T\) \cite{karabash2005covering}.
So the ``equilaterality'' of $T$ is essential for Conjecture \ref{conj:main} to be true \cite{conway2004coverup, soifer2009coffee}.
Also, Karabash and Soifer generalized Theorem \ref{thm:main-prev} and showed that a \emph{trigon}\footnote{A connected shape formed by unit equilateral triangles with matching edges.} made of \(n\) unit equilateral triangles can be covered by \(n + 2\) triangles of side \(1 - \varepsilon\) \cite{karabash2005covering}.
A similar problem of covering a square of side \(n + \varepsilon\) with unit squares has also been extensively studied \cite{chungEfficientPackingsUnit2020,chung2009packing,friedman2009packing,karabash2006sharp,karabash2008note,nagamochi2005packing}.
Still, to the best of the authors' knowledge, the original Conjecture \ref{conj:main} raised by Conway and Soifer hasn't been addressed directly in the literature.

Observe that in the constructions of Conway and Soifer (Figures \ref{fig:triangle1} and \ref{fig:triangle2}), all unit triangles are homothetic ($\bigtriangleup$ or $\bigtriangledown$) to $T$.
The generalized covering of trigons by Karabash and Soifer \cite{karabash2005covering} only uses triangles homothetic to $T$ as well.
Motivated by this, we show the following.

\begin{theorem}
If $T$ is an equilateral triangle of side \(> n\), then \(n^2 + 1\) unit equilateral triangles homothetic to $T$ cannot cover $T$.
\label{thm:triangle-cover-cor}
\end{theorem}

Note that Theorem \ref{thm:triangle-cover-cor} does not solve Conjecture \ref{conj:main}, as Theorem \ref{thm:triangle-cover-cor} restricts the rotation of unit equilateral triangles ($\bigtriangleup$ or $\bigtriangledown$) while Conjecture \ref{conj:main} allows arbitrary rotations.
Still, our theorem generalizes to triangles with a side parallel to the $x$-axis which may not be equilateral (Theorem \ref{thm:triangle-cover}).

\begin{definition}
An \emph{$H$-triangle} (a shorthand notation for a \emph{horizontal triangle}) is a triangle with one side parallel to the $x$-axis.
For any $H$-triangle $T$, its \emph{base} is the length of the side $l$ parallel to the $x$-axis, and its \emph{height} is the distance between $l$ and the vertex of $T$ which is not on $l$. 
\end{definition}

\begin{theorem}
Let \(X\) be any union of \(n\) $H$-triangles of base $b$ and height $h$ with disjoint interiors. Then \(X\) cannot be covered by \(n + 1\) $H$-triangles of base less than $b$ and height less than $h$.
\label{thm:triangle-cover}
\end{theorem}

Note that in Theorem \ref{thm:triangle-cover}, the $H$-triangles are not necessarily congruent or similar to each other.
To recover Theorem \ref{thm:triangle-cover-cor} from Theorem \ref{thm:triangle-cover}, assume that an equilateral $H$-triangle \(T\) of side \(> n\) can be covered by \(n^2 + 1\) unit equilateral $H$-triangles.
Shrink the covering so that \(T\) has sides exactly \(n\) and the small triangles have sides \(< 1\).
Then we get a contradiction by Theorem \ref{thm:triangle-cover} as \(T\) is a union of \(n^2\) unit equilateral $H$-triangles with disjoint interiors.

As the coverings of \(T\) by Conway and Soifer (Figures \ref{fig:triangle1} and \ref{fig:triangle2}) and the coverings of trigons by Karabash and Soifer only use $H$-triangles,
we determine the exact minimum number of unit equilateral $H$-triangles required for covering $T$.

\begin{corollary}
The minimum number of unit equilateral $H$-triangles required to cover an equilateral $H$-triangle of side \(n + \varepsilon\) with a sufficiently small $\varepsilon > 0$ is \(n^2+2\).

Also, the minimum number of unit equilateral $H$-triangles required to cover a trigon made of \(n\) equilateral $H$-triangles of side \(1 + \varepsilon\) with a sufficiently small $\varepsilon > 0$ is \(n + 2\).
\label{cor:triangle-cover-number}
\end{corollary}

We can ask for the maximum possible $\varepsilon > 0$ such that the equilateral triangle $T$ of side $n + \varepsilon$ can be covered by $n^2+2$ unit triangles.
From area considerations, we get $(n + \varepsilon)^2 \geq n^2+2$ and the trivial upper bound 
$\varepsilon \leq \sqrt{n^2+2} - n = 1/n - 1/(2n^3) + O(1/n^5)$.
If we require all triangles to be $H$-triangles, then our method can be used to determine the exact maximum value of $\varepsilon$.

\begin{theorem}
The largest value of $\varepsilon > 0$ such that the equilateral $H$-triangle $T$ of side $n + \varepsilon$ can be covered by $n^2+2$ equilateral $H$-triangles is $\varepsilon = 1/(n+1)$.
\label{thm:max-epsilon}
\end{theorem}
This maximum value $\varepsilon = 1/(n+1)$ is achieved by the first construction of Conway and Soifer \cite{conway2004coverup} (Figure \ref{fig:triangle1}); see Section \ref{sec:description}.
We also consider the same question with $n^2 + 3$ triangles and obtain the following result.
\begin{theorem}
The largest value of $\varepsilon > 0$ such that the equilateral $H$-triangle $T$ of side $n + \varepsilon$ can be covered by $n^2+3$ equilateral $H$-triangles is $\varepsilon = 1/n$.
\label{thm:max-epsilon-3}
\end{theorem}

The proof of Theorems \ref{thm:triangle-cover}, \ref{thm:max-epsilon}, and \ref{thm:max-epsilon-3} are based on analyzing specific properties of an abelian group $\mathcal{T}$ (Definition \ref{def:t}) of functions from $[0, 1)$ to $\mathbb{R}$.
To the best of the authors' knowledge, such a method is entirely new for understanding covering problems.

\section{A description of two coverings by Conway and Soifer.}
\label{sec:description}

Before proving the main theorems (Theorems \ref{thm:triangle-cover}, \ref{thm:max-epsilon}, and \ref{thm:max-epsilon-3})
we describe the constructions of Conway and Soifer (Figures \ref{fig:triangle1} and \ref{fig:triangle2}) in detail that prove Theorem \ref{thm:main-prev}.
Readers interested in the main proofs of the paper can jump right to Section \ref{sec:main1}.

In Figure \ref{fig:triangle1}, we first cover the upper part of $T$ which is an equilateral triangle of side length $n-1$ with $(n-1)^2$ triangles (filled with small circles).
The remaining part is a trapezoid of side lengths $1+\varepsilon$, $n+\varepsilon$, $1+\varepsilon$, and $n-1$.
Now interleave $2n-1$ triangles from the right to cover the trapezoid (white triangles).
We can check that the remaining part is a parallelogram of side lengths $1+\varepsilon$ and $n \varepsilon$,
subtracted by a small equilateral triangle of length $\varepsilon$ on the right-upper corner.
This parallelogram can be covered with two triangles if $\varepsilon \leq 1 / (n + 1)$ (filled with rhombi).
The picture depicts the maximum case $\varepsilon = 1 / (n + 1)$ for $n = 4$.

In Figure \ref{fig:triangle2}, we cover the large triangle $T$ from the bottom.
We first cover the bottommost layer with $n$ upward triangles and $n-1$ downward triangles,
with each triangle misaligned from neighboring triangles by $\varepsilon' = \varepsilon / (n-1)$ (filled with large circles).
The covered trapezoid has side lengths 
$$
1 - \varepsilon', n +(n-1)\varepsilon' = n + \varepsilon, 1 - \varepsilon', n - 1 + n\varepsilon'
$$
with small `bumps' of length $\varepsilon'$ from triangles directing upwards.
We then stack the next bottommost layer (filled with rhombi) with $n-1$ upward triangles and $n-2$ downward triangles misaligned by $\varepsilon''$.
To cover the upper side of the large-circle-patterned trapezoid tightly, our new $\varepsilon''$ should satisfy 
$$(n-1) + (n-2)\varepsilon'' = (n-1) + n\varepsilon'$$
hence $\varepsilon'' = n\varepsilon' / (n-2)$.
We continue this until we stack the total of $(n-1)$ layers, where the topmost layer (filled with small circles) consists of two upward triangles and one downward triangle with deviation
$$
\frac{n}{n-2} \frac{n-1}{n-3} \frac{n-2}{n-4} \cdots \frac{3}{1} \varepsilon' = \frac{n(n-1)}{2} \varepsilon' = \frac{n}{2} \varepsilon.
$$
The remaining part of the triangle can be covered with three triangles of unit lengths (white triangles) if $1 + 2\cdot \frac{n}{2}\varepsilon \le 3/2$, or equivalently, if $\varepsilon \leq 1/(2n)$.
The figure depicts the maximal case $\varepsilon = 1/(2n)$ for $n = 4$.

\section{Proof of Theorem \ref{thm:triangle-cover}.}
\label{sec:main1}

We now prove Theorem \ref{thm:triangle-cover}.
By rescaling the coordinates, we can assume that both the base $b$ and the height $h$ are equal to 1 without loss of generality.
We will use the following notion frequently.

\begin{definition}
For every $H$-triangle \(T\), define its \(y\)-coordinate \(y_T\) as the \(y\)-coordinate of the horizontal side of $T$.
\end{definition}

For every $H$-triangle $T$, we define a function $f_T : [0, 1) \rightarrow \mathbb{R}$ that will be the main ingredient for our proof.

\begin{definition}
For every $H$-triangle $T$ and $t \in \mathbb{R}$, 
define $\tilde{f}_T(t)$ as the length of the segment of the line \(y = t\) covered by \(T\) unless $t = y_T$ and the line contains the base.
For $t = y_T$, choose the value of \(\tilde{f}_T(y_T)\) so that \(\tilde{f}_T\) is right-continuous: the base of $T$ if \(T\) is pointed upwards, and 0 if \(T\) is pointed downwards.
Define \(f_T : [0, 1) \to \mathbb{R}\) as the function \(f_T(t) = \sum_{n \in \mathbb{Z}} \tilde{f}_T(t + n)\).
\label{def:vertical-triangle-function}
\end{definition}

For any real number \(x\), let \(\left\{ x \right\}\) be the value in \([0, 1)\) equal to \(x\) modulo 1. If $T$ is an $H$-triangle with base 1 and height 1, we can characterize all possibilities of $f_T$ as the following.

\begin{corollary}
If an $H$-triangle \(T\) of base 1 and height 1 is pointed downwards, then \(f_T(t) = \{t - y_T\}\), and if \(T\) is pointed upwards, then \(f_T(t) = 1 - \{t - y_T\}\).
\label{cor:unit-projection-ftn}
\end{corollary}

\begin{figure}[h]
  \begin{minipage}{0.5\textwidth}
    \centering
    \includegraphics[width=0.9\linewidth]{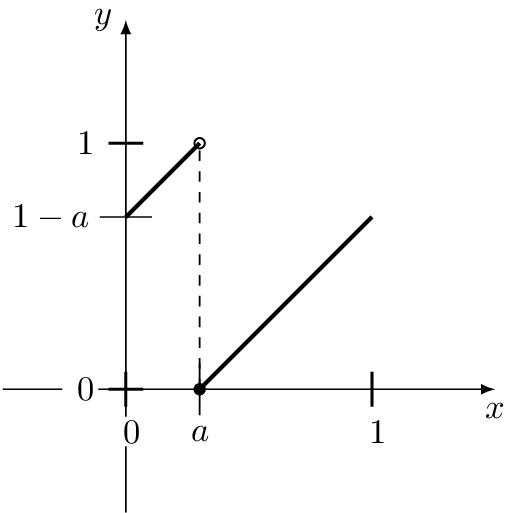}
  \end{minipage}\hfill
  \begin{minipage}{0.5\textwidth}
    \centering
    \includegraphics[width=0.9\linewidth]{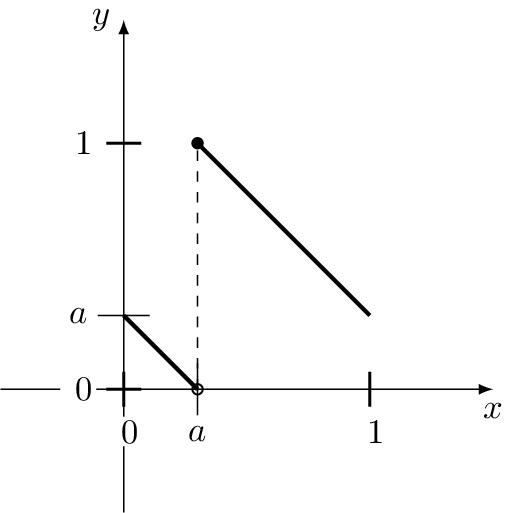}
  \end{minipage}
  \label{fig:graph1}
  \caption{Graphs of $t \mapsto \{t - a\}$ and $t \mapsto 1 - \{t - a\}$ for $a = 0.3$.}
\end{figure}

We proceed with the proof of Theorem \ref{thm:triangle-cover}.
Assume that the union \(X\) of \(n\) $H$-triangles \(S_1, S_2, \dots, S_n\) with base 1, height 1, and disjoint interiors can be covered by \(n+1\) $H$-triangles \(T'_0, T'_1, \dots, T'_n\) of base and height \(< 1\).
For each $T_i'$, take an arbitrary $H$-triangle $T_i$ of base 1 and height 1 so that \(T_i\) contains \(T_i'\).

Define \(\tilde{g} : \mathbb{R} \to \mathbb{R}\) as the function \(\tilde{g} = \sum_{i=0}^n \tilde{f}_{T_i} - \sum_{j=1}^n \tilde{f}_{S_j}\). Take any \(t\) different from the $y$-coordinates \(y_{T_i}\) and \(y_{S_j}\) of the triangles. As the triangles \(T_0, T_1, \dots, T_n\) cover the union \(X\) of disjoint triangles \(S_1, S_2, \dots, S_n\), the total length of the portion of the line \(y = t\) covered by \(T_i\)'s is at least the total length of the parts of the line \(y = t\) covered by \(S_j\)'s. Thus, we have \(\tilde{g}(t) \geq 0\). As \(\tilde{g}\) is right-continuous, by sending the right limit, we have \(\tilde{g}(t) \geq 0\) for every \(t \in \mathbb{R}\), including the case where \(t\) is equal to the \(y\)-coordinate of some triangle.

Define \(g : [0, 1) \to \mathbb{R}\) as \(g = \sum_{i=0}^n f_{T_i} - \sum_{j=1}^n f_{S_j}\) so that we have \(g(t) = \sum_{n \in \mathbb{Z}} \tilde{g}(t + n)\). Then, consequently, we have \(g(t) \geq 0\) for every \(t \in [0, 1)\). 
It turns out that this is sufficient to derive a contradiction. 

The following group is the key to our proofs.
\begin{definition}
\label{def:t}
Define \(\mathcal{T}\) as the abelian group generated by all functions \(t \mapsto \{ t - a\}\) and \(t \mapsto 1 - \{t - a\}\) with \(a \in [0, 1)\).
\end{definition}
In other words, $\mathcal{T}$ is the set of all functions from $[0, 1)$ to $\mathbb{R}$ that can be expressed as a finite addition and subtraction of functions of form \(t \mapsto \{ t - a\}\) or \(t \mapsto 1 - \{t - a\}\) with arbitrary \(a \in [0, 1)\).
Then \(g \in \mathcal{T}\) by Corollary \ref{cor:unit-projection-ftn}.

We now examine the properties of $g \in \mathcal{T}$.
\begin{definition}
Denote the integral of any integrable function \(f : [0, 1) \to \mathbb{R}\) over the whole \([0, 1)\) as simply \(\int f\).
\end{definition}

\begin{lemma}

Any function \(f : [0, 1) \to \mathbb{R}\) in \(\mathcal{T}\) has the following properties.

\begin{enumerate}[(i)]
\item
  \(f\) is right-continuous.
\item
  \(f\) is differentiable everywhere except for a finite number of points, and the derivative is always equal to a fixed constant \(a \in \mathbb{Z}\).
\item
  For all \(s, t \in [0, 1)\), the value \(f(t) - f(s)\) is equal to \(a(t - s)\) modulo 1.
\item
  The integral \(\int f\) is equal to \(b / 2\) for some \(b \in \mathbb{Z}\) where \(b - a\) is divisible by 2.
\end{enumerate}

\label{lem:triangle-group}
\end{lemma}

\begin{proof}
Check that all the claimed properties are closed under addition and negation. Then, check that the functions \(t \mapsto \{t - c\}\) and \(t \mapsto 1 - \{t - c\}\) with \(c \in [0, 1)\) satisfy the claimed properties.
\end{proof}

We observed that \(g \in \mathcal{T}\) and \(g(t) \geq 0\) for every \(t \in [0, 1)\). Also, for any $H$-triangle \(T\) of base 1 and height 1, we have \(\int f_T = 1/2\), so we also have \(\int g = 1/2\) by the definition \(g = \sum_{i=0}^n f_{T_i} - \sum_{j=1}^n f_{S_j}\). 
We now use the following lemma.

\begin{lemma}

Let \(f : [0, 1) \to \mathbb{R}\) be any function in \(\mathcal{T}\) such that \(\int f = 1/2\) and \({f(t) \geq 0}\) for every \(t \in [0, 1)\). Then there is a positive odd integer \(a\) and some \(c \in [0, 1)\) such that \(f\) is either \(f(t) = \left\{ at + c \right\}\) or \(f(t) = 1 - \left\{ at + c \right\}\).

\label{lem:triangle-unit-area}
\end{lemma}

\begin{figure}[h]
  \begin{minipage}{0.5\textwidth}
    \centering
    \includegraphics[width=0.9\linewidth]{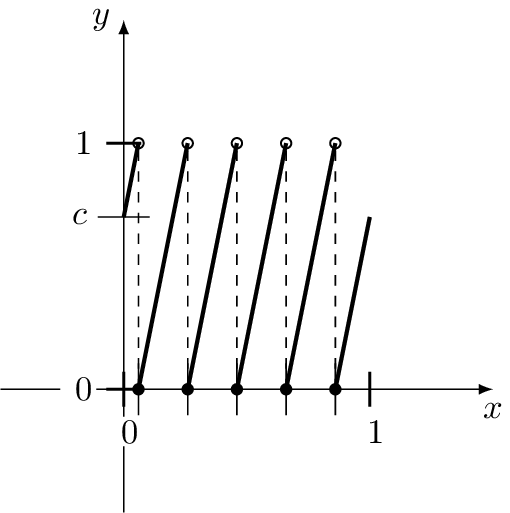}
  \end{minipage}\hfill
  \begin{minipage}{0.5\textwidth}
    \centering
    \includegraphics[width=0.9\linewidth]{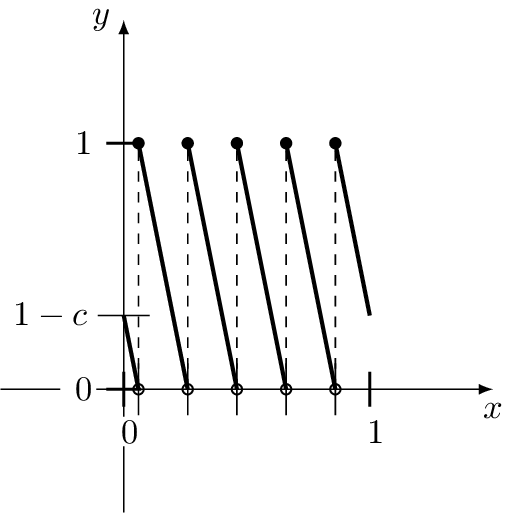}
  \end{minipage}
  \caption{Graphs of $t \mapsto \{at + c\}$ and $t \mapsto 1 - \{at + c\}$ for $a = 5$ and $c = 0.7$.}
  \label{fig:graph2}
\end{figure}

\begin{proof}
By Lemma \ref{lem:triangle-group}, there is some odd number \(a \in \mathbb{Z}\) such that \(f'(t) = a\) for all \(t\) except for a finite number of values. Let \(f(0) = c\), then by Lemma \ref{lem:triangle-group} again, we have \(f(t)\) equal to \(at + c\) modulo 1 for all \(t \in [0, 1)\). Let \(g : [0, 1) \to \mathbb{R}\) be the function \(g(t) = \{at + c\}\).
Then, for every \(t \in [0, 1)\), as the value \(f(t)\) is nonnegative and equal to \(at + c\) modulo 1, we have \(f(t) \geq g(t) \geq 0\). But note that the integral \(\int g\) is exactly equal to \(1/2\) (see Figure \ref{fig:graph2}). So \(f\) and \(g\) should be equal almost everywhere. As \(f\) is right-continuous by Lemma \ref{lem:triangle-group}, \(f(t)\) should be equal to the right limit \(g(t +) = \lim_{u \to t+} g(u)\) of \(g\).
If \(a > 0\), then \(g\) is right-continuous, so \(f(t) = g(t) = \left\{ at + c \right\}\). If \(a < 0\), then the right limit of \(g\) is \(1 - \left\{ -at + \left\{ - c \right\} \right\}\) (this is the value in \((0, 1]\) equal to \(at + c\) modulo 1).
\end{proof}

We now finish the proof of Theorem \ref{thm:triangle-cover}. By Lemma \ref{lem:triangle-unit-area}, the discontinuities of \(g = \sum_{i=0}^n f_{T_i} - \sum_{j=1}^n f_{S_j}\) have to be equidistributed in \([0, 1)\) with a gap of \(1/a\) for some positive odd number \(a\).
But each \(T_i\) can be taken arbitrarily as long as it contains the smaller triangle \(T_i'\) of side \(< 1\). So take each \(T_i\) so that the \(y\)-coordinates \(y_{T_0}, y_{T_1}, \dots, y_{T_n}\) are nonzero and different from \(y_{S_1}, y_{S_2}, \dots, y_{S_n}\) modulo 1.
Also, we can wiggle $T_1$ a bit to make \(y_{T_1} - y_{T_0}\) an irrational number.
Then \(g\) has discontinuities at \(\{y_{T_0}\}, \{y_{T_1}\}, \dots, \{y_{T_n}\} \in [0, 1)\), and two of them have an irrational gap. This gives contradiction and finishes the proof.

\section{Proof of Theorems \ref{thm:max-epsilon} and \ref{thm:max-epsilon-3}.}
\label{sec:max-epsilon}

Using the group $\mathcal{T}$ in Definition \ref{def:t} was the key idea of the proof of Theorem \ref{thm:triangle-cover}.
We use the same idea to determine the maximum $\varepsilon > 0$ such that the equilateral $H$-triangle $T$ of side $n + \varepsilon$ can be covered by $n^2+2$ or $n^2+3$ unit equilateral $H$-triangle respectively (Theorems \ref{thm:max-epsilon} and \ref{thm:max-epsilon-3}).

We first construct the optimal coverings.
The analysis in Section \ref{sec:description} shows that Figure \ref{fig:triangle1} is a covering of $T$ with $n^2+2$ $H$-triangles for $\varepsilon = 1/(n+1)$.
It can be modified to coverings of $T$ with $n^2 + 3$ unit $H$-triangles for $\varepsilon = 1/n$
as well (Figure \ref{fig:triangle3}).
In the last row of Figure \ref{fig:triangle1}, replace any three adjacent triangles with an equilateral $H$-triangle of side 2 made out of four unit equilateral $H$-triangles.

\begin{figure}[h]
  \centering
  \includegraphics[width=0.9\linewidth]{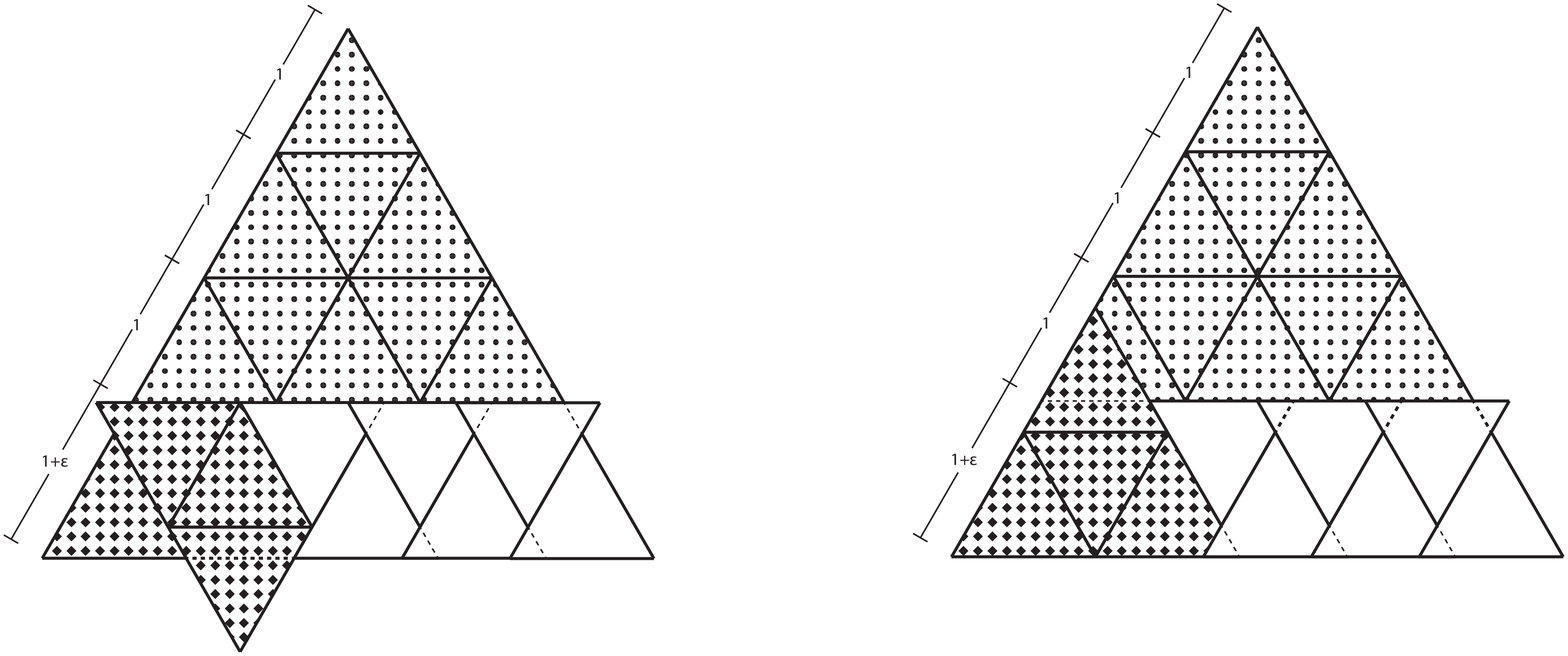}
  \caption{Two coverings of a equilateral $H$-triangle of length $n + 1/n$ with $n^2 + 3$ unit triangles ($n = 4$).}
  \label{fig:triangle3}
\end{figure}

Before proceeding further, we prove a corollary of Lemma \ref{lem:triangle-unit-area} to use.
\begin{corollary}
\label{cor:jump}
If $g, h \in \mathcal{T}$ and there is some $\delta > 0$ such that $g(t) + \delta \leq h(t)$ for all $t \in [0, 1)$, then $\int g + 1 \leq \int h$.
\end{corollary}
\begin{proof}
Let $f = h - g \in \mathcal{T}$. As $\delta \leq f(t)$ for all $t \in [0, 1)$, we have $0 < \int f$. 
Because $\int f$ is always a half-integer, the only possible case where $\int f < 1$ is when $\int f = 1/2$.
But by Lemma \ref{lem:triangle-unit-area}, such an $f(t)$ cannot satisfy $\delta \leq f(t)$ for all $t \in [0, 1)$, which leads to a contradiction.
So we have $1 \leq \int f$.
\end{proof}

We now show that if $\varepsilon > 1/(n+1)$ (resp. $\varepsilon > 1/n$), then it is impossible to cover $T$ with $n^2+2$ (resp. $n^2+3$) unit equilateral $H$-triangles.
Assume by contradiction that a covering with $N = n^2+2$ or $n^2+3$ unit equilateral $H$-triangles $S_1, S_2, \dots, S_N$ exists.
Stretch the covering vertically by a factor of $2 / \sqrt{3}$ so that each $S_i$ has base and height 1, and $T$ has base and height $n + \varepsilon$.
In this way, each function $f_{S_i}$ satisfies the condition of Corollary 2.
Without loss of generality, assume that the bottom side of $T$ is the $x$-axis so that $y_T = 0$.
Let $f_S = \sum_{i=1}^N f_{S_i}$.
We will derive a contradiction from $f_T \leq f_S$.

Define $T_0$ as the equilateral $H$-triangle of side $n$ with $y_{T_0}=0$ pointed upwards, sharing the leftmost vertex with $T$ on the line $y=0$.
Define $r = f_{T_0} + 1$, then $r \in \mathcal{T}$, and we have $\int r = (n^2+2)/2$.
Our strategy is to compare $f_T$ and $f_S$ using $r \in \mathcal{T}$ as a reference.
Define $g = f_T - r$ and $h = f_S - r$, then we have $g \leq h \in \mathcal{T}$.

\begin{figure}[h]
  \centering
  \includegraphics[width=0.5\linewidth]{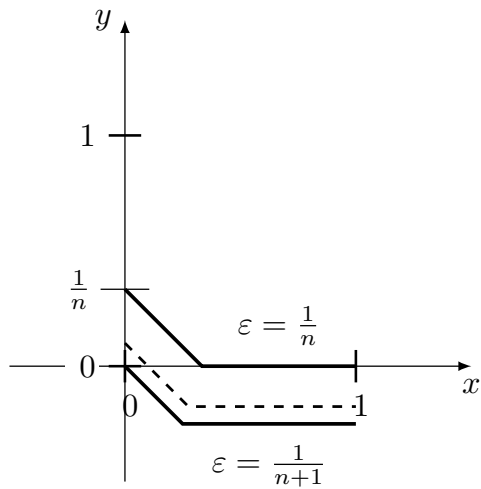}
  \caption{Graph of $g(t)$ for values of $\varepsilon$ between $1/(n+1)$ and $1/n$ ($n=3$). Bold lines are the graphs where $\varepsilon = 1/(n+1)$ or $1/n$, and dashed line is the graph for a general $\varepsilon$ in between.}
  \label{fig:graph-g}
\end{figure}

We now compute a lower bound of $g$.
Note that $T$ is obtained from $T_0$ by padding a parallelogram of base $\varepsilon$ and height $n$ (which is $(\sqrt{3}/2) n$ before stretching) to the right of $T_0$ and then putting a triangle of base and height $\varepsilon$ on top of the parallelogram. 
So by comparing $T$ with $T_0$, we have 
$$
f_T(t) \geq f_{T_0}(t) + n\varepsilon + \max(0, \varepsilon - t)
$$
where the equality holds for every $\varepsilon \leq 1$. So
$$g(t) \geq - (1 - n\varepsilon) + \max(0, \varepsilon - t)$$
by subtracting $r(t) = f_{T_0}(t) + 1$ from $f_T(t)$ (see Figure \ref{fig:graph-g}).

If $N = n^2+2$ and $\varepsilon > 1/(n+1)$, then we have 
$$
h(t) \geq g(t) \geq -(1 - n \varepsilon) + \varepsilon -t = -t + ((n + 1) \varepsilon - 1)
$$
for all $0 \leq t < 1$.
So by Corollary \ref{cor:jump} applied for the functions $-t$ and $h$ with $\delta = (n + 1) \varepsilon - 1 > 0$, we have $1/2 \leq \int h$, and this contradicts our assumption that $\int h = \int f_S - \int r = 0$.
This proves Theorem \ref{thm:max-epsilon}.

If $N = n^2+3$ and $\varepsilon > 1/n$, then we have
$$
h(t) \geq g(t) \geq n \varepsilon - 1
$$ 
for all $0 \leq t < 1$.
By comparing $0$ and $h$ using Corollary \ref{cor:jump} with $\delta = n \varepsilon - 1$, we have $1 \leq \int h$, and this contradicts our assumption that $\int h = \int f_S - \int r = 1/2$.
This proves Theorem \ref{thm:max-epsilon-3}.

\section{Conclusion and Remarks}
A conjecture by Conway and Soifer states that an equilateral triangle of side $> n$ cannot be covered by $n^2+1$ unit equilateral triangles (Conjecture \ref{conj:main}).
We made partial progress towards their conjecture by showing that $n^2+1$ unit equilateral triangles with a side parallel to the $x$-axis ($\bigtriangleup$ or $\bigtriangledown$) cannot cover an equilateral triangle of side $> n$ parallel to the $x$-axis (Theorem \ref{thm:triangle-cover-cor}).

Our method analyzes an abelian group $\mathcal{T}$ (Definition \ref{def:t}) of piecewise-linear functions.
The method generalizes to triangles with a side parallel to the $x$-axis that may not necessarily be equilateral (Theorem \ref{thm:triangle-cover}).
In particular, for any $b, h > 0$, a triangle of base $> nb$ parallel to the $x$-axis and height $> nh$ cannot be covered by $n^2+1$ triangles, each with a base $b$ parallel to the $x$-axis and height $h$.

A natural strengthening of Conjecture \ref{conj:main} is to find the largest side of an equilateral triangle that can be covered by $n^2 + k$ unit equilateral triangles for $1 \leq k \leq 2n$.
\begin{question}
\label{que:largest}
What is the largest side length of an equilateral triangle that can be covered by $n^2 + k$ unit equilateral triangles for $1 \leq k \leq 2n$?
\end{question}
Using the same method of analyzing the abelian group $\mathcal{T}$, we were able to answer the following variant of Question \ref{que:largest} with $k = 2, 3$ (Theorems \ref{thm:max-epsilon} and \ref{thm:max-epsilon-3}).
\begin{question}
\label{que:largest-h}
What is the answer to Question \ref{que:largest} if the unit equilateral triangles are required to have a side parallel to the $x$-axis?
\end{question}
The readers are encouraged to make further progress towards Question \ref{que:largest-h} with other values of $k$ or the full Conjecture \ref{conj:main}.

A `dual' version of Question \ref{que:largest} is to find the \emph{minimum} side length of an equilateral triangle in which we can \emph{pack} $n^2 - k$ unit equilateral triangles for $1 \leq k \leq 2n - 2$.
\begin{question}
\label{que:largest-square}
What is the minimum side length of an equilateral triangle in which we can pack $n^2 - k$ unit equilateral triangles for $1 \leq k \leq 2n - 2$?
\end{question}
Note that a related problem of packing unit squares inside a square has been studied extensively.
\begin{question}
\label{que:smallest-square}
What is the minimum side length of a square in which we can pack $n^2 - k$ unit squares for $1 \leq k \leq 2n - 2$?
\end{question}
Erich Friedman \cite{friedman2009packing} gives a comprehensive survey of known results on Question \ref{que:smallest-square}. Hiroshi Nagamochi \cite{nagamochi2005packing} answered Question \ref{que:smallest-square} for $k=1, 2$ by showing that the minimum side length of a square in which $n^2 - 2$ or $n^2 - 1$ unit squares can be packed is exactly $n$, akin to Conjecture \ref{conj:main}.

\begin{acknowledgment}{Acknowledgment.}
We thank Alexander Soifer for valuable comments, and the anonymous reviewers for their insightful feedback that greatly improved our draft.
Jineon Baek acknowledges support from the Korea Foundation for Advanced Studies (KFAS) for graduate fellowship.
\end{acknowledgment}

\begin{biog}
  \item[Jineon Baek] is a Ph.D. candidate in mathematics at the University of Michigan--Ann Arbor. He received his B.S. in mathematics from the Pohang University of Science and Technology (POSTECH). His main research interests are in optimization problems of combinatorics and geometry.
  \begin{affil}
  Department of Mathematics, University of Michigan--Ann Arbor, Ann Arbor, MI 48109-1043\\
  jineon@umich.edu
  \end{affil}
  
  \item[Seewoo Lee] is a Ph.D. candidate in mathematics at the University of California--Berkeley. He received his B.S. and M.S. in mathematics from the Pohang University of Science and Technology (POSTECH).
  \begin{affil}
  Department of Mathematics, University of California--Berkeley, Berkeley, CA 94720\\
  seewoo5@berkeley.edu
  \end{affil}
  \end{biog}
  \vfill\eject

\end{document}